\documentclass[a4paper,12pt]{article}
\usepackage{amsmath}
\usepackage{amsthm}
\usepackage{amssymb}
\usepackage{enumerate}
\usepackage{url}
\usepackage[utf8]{inputenc}
\usepackage{comment}

\title{Invariance of Brownian motion associated with past and future maxima}

\author{Yuu Hariya\thanks{Supported in part by JSPS KAKENHI Grant Number  22K03330}}

\date{\empty}
\numberwithin{equation}{section}

\theoremstyle{plain}

\newtheorem{thm}{Theorem}[section]
\newtheorem{lem}[thm]{Lemma}
\newtheorem{prop}[thm]{Proposition}
\newtheorem{cor}[thm]{Corollary}

\theoremstyle{definition}

\theoremstyle{remark}
\newtheorem{rem}[thm]{Remark}

\begin{document}

\newcommand\ND{\newcommand}
\newcommand\RD{\renewcommand}

\ND\N{\mathbb{N}}
\ND\R{\mathbb{R}}
\ND\Q{\mathbb{Q}}
\ND\C{\mathbb{C}}

\ND\F{\mathcal{F}}

\ND\kp{\kappa}

\ND\ind{\boldsymbol{1}}

\ND\al{\alpha }
\ND\la{\lambda }
\ND\La{\Lambda }
\ND\ve{\varepsilon}

\ND\ga{\gamma}

\ND\lref[1]{Lemma~\ref{#1}}
\ND\tref[1]{Theorem~\ref{#1}}
\ND\pref[1]{Proposition~\ref{#1}}
\ND\sref[1]{Section~\ref{#1}}
\ND\ssref[1]{Subsection~\ref{#1}}
\ND\aref[1]{Appendix~\ref{#1}}
\ND\rref[1]{Remark~\ref{#1}} 
\ND\cref[1]{Corollary~\ref{#1}}
\ND\eref[1]{Example~\ref{#1}}
\ND\fref[1]{Fig.\ {#1} }
\ND\lsref[1]{Lemmas~\ref{#1}}
\ND\tsref[1]{Theorems~\ref{#1}}
\ND\dref[1]{Definition~\ref{#1}}
\ND\psref[1]{Propositions~\ref{#1}}
\ND\rsref[1]{Remarks~\ref{#1}}
\ND\sssref[1]{Subsections~\ref{#1}}

\ND\pr{\mathbb{P}}
\ND\ex{\mathbb{E}}

\ND\cA{\Tilde{A}}

\ND\eqd{\stackrel{(d)}{=}}
\ND\db[1]{B^{(#1)}}

\ND\tr{\mathbb{T}}

\ND\ct{\mathcal{T}}

\ND\ctc{\mathcal{T}^{(c)}}

\ND\cg{\mathcal{G}}

\ND\cm{\mathcal{M}}

\ND\cp{\mathcal{P}}

\ND\cs{\mathcal{S}}

\ND\bc{B^{c}}

\ND\T{T}

\ND\bes{r}

\ND\id{\mathrm{Id}}

\ND{\rmid}[1]{\mathrel{}\middle#1\mathrel{}}

\def\thefootnote{{}}

\maketitle 
\begin{abstract}
Let $B=\{ B_{t}\} _{t\ge 0}$ be a one-dimensional standard Brownian motion. As an application of a recent result of ours on exponential functionals of Brownian motion, we show in this paper that, for every fixed $t>0$, the process given by 
\begin{align*}
 B_{s}-B_{t}-\Bigl| 
 B_{t}+\max _{0\le u\le s}B_{u}-\max _{s\le u\le t}B_{u}
 \Bigr| 
 +\Bigl| 
 \max _{0\le u\le s}B_{u}-\max _{s\le u\le t}B_{u}
 \Bigr| ,\quad 0\le s\le t,
\end{align*}
is a Brownian motion. The path transformation that describes the above process is proven to be an involution, commute with time reversal, and preserve Pitman's transformation. A connection with Pitman's $2M-X$ theorem is also discussed.
\footnote{Mathematical Institute, Tohoku University, Aoba-ku, Sendai 980-8578, Japan}
\footnote{E-mail: hariya@tohoku.ac.jp}
\end{abstract}

\section{Introduction and main results}\label{;intro}

Let us consider a one-dimensional standard Brownian $B=\{ B_{t}\} _{t\ge 0}$ 
along with its exponential additive functional 
$A_{t},\,t\ge 0$, defined by 
\begin{align*}
 A_{t}\equiv A_{t}(B):=\int _{0}^{t}e^{2B_{s}}\,ds.
\end{align*}
This additive functional of Brownian motion has importance in a number of 
areas of probability theory such as stochastic analysis of Laplacians on 
hyperbolic spaces and pricing of Asian options in mathematical finance, and 
is also known for its close relationship with planar Brownian motion; see, e.g., 
the detailed surveys~\cite{mySI, mySII} by Matsumoto and Yor.

For every fixed $t>0$, we observed in \cite{har22+} that the Wiener measure 
on the space $C([0,t];\R )$ of real-valued continuous functions over $[0,t]$ 
is invariant under the anticipative transformation 
$\ct \equiv \ct^{t}: C([0,t];\R )\to C([0,t];\R )$ defined by  
\begin{align}\label{;defct}
 \ct (\phi )(s):=\phi _{s}-\log \left\{ 
 1+\frac{A_{s}(\phi )}{A_{t}(\phi )}\left( e^{2\phi _{t}}-1\right) 
 \right\} ,\quad 0\le s\le t,\ \phi \in C([0,t];\R ).
\end{align}
Namely, it holds that  
\begin{align}\label{;invw}
 \{ \ct (B)(s)\} _{0\le s\le t}\eqd \{ B_{s}\} _{0\le s\le t}.
\end{align}
See \cite[Theorem~1.1]{har22+}; we also append a proof for the reader's 
convenience in the appendix. Anticipative transforms of Brownian motion 
have been studied by many authors particularly in the framework of 
Malliavin calculus; one of their main concerns is absolute continuity relationships, 
or Girsanov-type formulas, for those transformations with respect to the 
underlying Wiener measure. We refer the reader to \cite[Section~6]{har22} 
and references therein for more accounts of this topic. The transformation 
$\ct $ defined in \eqref{;defct} is an instance in which we can explicitly calculate 
the corresponding Radon--Nikod\'ym density, which, in fact, turns out to be 
identically equal to $1$: for any nonnegative measurable function $F$ on 
$C([0,t];\R )$, 
\begin{align*}
 \ex \!\left[ 
 F(\ct (B))
 \right] =\ex \!\left[ 
 F(B)
 \right] .
\end{align*}
Path transformations to be studied in the sequel serve 
as another instance of such a situation. Here and below, we say that 
a function on $C([0,t];\R )$ is measurable if it is Borel-measurable with respect 
to the topology of uniform convergence in $C([0,t];\R )$.

The aim of the present paper is to understand better the aforementioned 
invariance in law of Brownian motion, which we will accomplish by incorporating a 
positive real parameter $c$ into the definition~\eqref{;defct} of $\ct $ to consider 
two extremals by tending $c\downarrow 0$ and $c\to \infty $.
In what follows, unless otherwise specified, we fix $t>0$. For every $c>0$, define 
$\ctc \equiv \ct ^{t,(c)}:C([0,t];\R )\to C([0,t];\R )$ by 
\begin{align}\label{;defctc}
 \ctc (\phi ):=\frac{1}{c}\ct (c\phi ),\quad \phi \in C([0,t];\R ).
\end{align}
Since $\ct $ is an involution and commutes with time reversal,  
each $\ctc $ is readily seen to enjoy these two properties as well; 
see \sref{;sptranss} for details. One of the main results of the paper 
is the following generalization of \eqref{;invw}:

\begin{thm}\label{;tmain1}
 For every $c>0$, we have 
 \begin{align*}
  \bigl\{ \ctc (B)(s)\bigr\} _{0\le s\le t}\eqd \{ B_{s}\} _{0\le s\le t}.
 \end{align*}
\end{thm}

In order to describe two extremal cases of the above theorem with 
$c\downarrow 0$ and $c\to \infty $, we introduce two transformations 
on $C([0,t];\R )$ denoted respectively by $\cg \equiv \cg ^{t}$ and 
$\cm \equiv \cm ^{t}$: for each $\phi \in C([0,t];\R )$,
\begin{align}
 \cg (\phi )(s)&:=\phi _{s}-\frac{2s}{t}\phi _{t}, \label{;defcg}\\
 \cm (\phi )(s)&:=\phi _{s}-\phi _{t}-\Bigl| 
 \phi _{t}+\max _{0\le u\le s}\phi _{u}-\max _{s\le u\le t}\phi _{u}
 \Bigr| 
 +\Bigl| 
 \max _{0\le u\le s}\phi _{u}-\max _{s\le u\le t}\phi _{u}
 \Bigr| , \label{;defcm}
\end{align}
for $0\le s\le t$. Then the law of Brownian motion is also invariant 
under $\cg $ and $\cm $.

\begin{thm}\label{;tmain2}
 It holds that 
 \begin{align*}
  \{ \cg (B)(s)\} _{0\le s\le t}\eqd 
  \{ \cm (B)(s)\} _{0\le s\le t}\eqd 
  \{ B_{s}\} _{0\le s\le t}.
\end{align*}
\end{thm}

Notice that, since the continuous process $\{ \cg (B)(s)\} _{0\le s\le t}$ 
is a centered Gaussian process, the above invariance under $\cg $ 
is immediate by observing that, for every $s_{1},s_{2}\in [0,t]$,
\begin{align*}
 &\ex \!\left[ 
 \cg (B)(s_{1})\cg (B)(s_{2})
 \right] \\
 &=\ex \!\left[ 
 B_{s_{1}}B_{s_{2}}
 \right]
 -\frac{2s_{1}}{t}\ex \!\left[ 
 B_{s_{2}}B_{t}
 \right] 
 -\frac{2s_{2}}{t}\ex \!\left[ 
 B_{s_{1}}B_{t}
 \right] 
 +\frac{4s_{1}s_{2}}{t^{2}}\ex \!\left[ 
 B_{t}^{2}
 \right] \\
 &=\min \{ s_{1},s_{2}\} -\frac{2s_{1}}{t}s_{2}-\frac{2s_{2}}{t}s_{1}
 +\frac{4s_{1}s_{2}}{t^{2}}t\\
 &=\min \{ s_{1},s_{2}\} ;
\end{align*}
yet the fact that $\cg $ is an involution and is commutative with 
time reversal, as each $\ctc $ is, is noteworthy, and 
we put those two properties of $\cg $ in \pref{;ppropcg}.

Recall the well-known Pitman's transformation, which we denote 
by $\cp $, defined on the space $C([0,\infty );\R )$ of real-valued 
continuous functions $\phi $ over $[0,\infty )$ as 
\begin{align}\label{;defcp}
 \cp (\phi )(s):=2\max _{0\le u\le s}\phi _{u}-\phi _{s},\quad s\ge 0.
\end{align}
For every $\mu \in \R $, we denote by 
$\db{\mu }=\bigl\{ \db{\mu }_{s}:=B_{s}+\mu s\bigr\} _{s\ge 0}$ the 
Brownian motion with drift $\mu $. In addition to being an involution 
and commuting with time reversal, the transformation $\cm $ 
is also seen to preserve $\cp $ restricted to $C([0,t];\R )$ (\pref{;ppropcm}). 
Together with the fact that $\cm (B)(t)=-B_{t}$ 
(see \eqref{;bvs} below), the Cameron--Martin formula entails 
the following corollary to \tref{;tmain2}: 

\begin{cor}\label{;ctmain2}
 For every $\mu \in \R $, it holds that 
 \begin{equation}\label{;jointl}
  \begin{split}
  &\left\{ 
  \bigl( 
  \db{\mu }_{s}, \cm (\db{\mu })(s), \cp (\db{\mu })(s)
  \bigr) 
  \right\} _{0\le s\le t}\\
  &\eqd \left\{ 
  \bigl( 
  \cm (\db{-\mu })(s), \db{-\mu }_{s},\cp (\db{-\mu })(s)
  \bigr) 
  \right\} _{0\le s\le t}.
  \end{split}
 \end{equation}
\end{cor}

The case $\mu =0$ extends the second identity in law in \tref{;tmain2}. 
Moreover, as $t>0$ is arbitrarily fixed, the above corollary may also be regarded as 
an extension of the known fact 
(cf.\ \cite[Theorem~1.2${}_\mu $\thetag{ii}]{myPI}) that 
\begin{align}\label{;pinv}
 \bigl\{ 
 \cp (\db{\mu })(s) 
 \bigr\} _{s\ge 0}
 \eqd \bigl\{ 
 \cp (\db{-\mu })(s) 
 \bigr\} _{s\ge 0},
\end{align}
for which we also refer the reader to \rref{;rpinv}\thetag{2}.

In order to facilitate the reader's understanding of the 
transformation $\cm $, we will also reveal that, given 
$\phi \in C([0,t];\R )$, the transformed path $\cm (\phi )$, as well as 
$\phi $ itself, is expressed in terms of Pitman's transformation and 
the terminal value $\phi _{t}$. 

\begin{prop}\label{;pexpress}
Restricting $\cp $ to $C([0,t];\R )$, for every $\phi \in C([0,t];\R )$, we have 
\begin{align}
 \phi _{s}&=-\cp (\phi )(s)+\min \Bigl\{ 
 2\min _{s\le u\le t}\cp (\phi )(u),\cp (\phi )(t)+\phi _{t}
 \Bigr\} , \label{;iip}\\
 \cm (\phi )(s)&=-\cp (\phi )(s)+\min \Bigl\{ 
 2\min _{s\le u\le t}\cp (\phi )(u),\cp (\phi )(t)-\phi _{t}
 \Bigr\} , \label{;imp}
\end{align}
for all $0\le s\le t$.
\end{prop}

The former relation~\eqref{;iip} shows that we can reconstruct the original 
path $\phi $ from $\cp (\phi )$ and $\phi _{t}$, which is of independent interest. 
We will see in \sref{;sppexpress} that the invariance in law of Brownian motion 
under $\cm $ can also be explained by the above two relations~\eqref{;iip} and 
\eqref{;imp}, together with the celebrated Pitman's $2M-X$ theorem \cite{jwp} 
on one-dimensional Brownian motion and the three-dimensional Bessel process. 

Pitman's transformation and its generalizations play an important role 
in the study of eigenvalues of random matrices of a certain class, and 
their relationship with the so-called Littelmann path model in representation 
theory has been investigated; see \cite{bbo, bj} and references therein. 
There is also a renewed interest in Pitman's or Pitman-type transformations 
in the context of probabilistic studies of classical integrable systems 
such as box-ball systems and the Toda lattice; see a series of recent papers 
\cite{ckst, cs1, cs2, cst1, cst2} by Croydon et al. We expect that our results 
will provide a new insight about that profound transformation.

We give an outline of the paper. In \sref{;sprf}, we prove \tsref{;tmain1} and 
\ref{;tmain2}. In \sref{;sptranss}, we investigate properties of the path 
transformations $\ctc $, $\cg $ and $\cm $ in 
\psref{;ppropctc}, \ref{;ppropcg} and \ref{;ppropcm}, respectively, and see how 
those of $\cm $ entail \cref{;ctmain2}. 
In \sref{;sppexpress}, after giving a proof of \pref{;pexpress}, we 
explore a connection between Pitman's theorem and the second 
identity in law in \tref{;tmain2}. To make the paper self-contained, we summarize 
some properties of the path transformation $\ct $ to be used in the proof 
of \pref{;ppropctc}, and provide their proofs in the appendix, in which 
we also give a brief proof of identity \eqref{;invw}. Two fundamental 
lemmas in calculus are appended as well.

\section{Proofs of \tsref{;tmain1} and \ref{;tmain2}}\label{;sprf}

In this section, we prove \tsref{;tmain1} and \ref{;tmain2}.

\subsection{Proof of \tref{;tmain1}}\label{;ssprf1}

In this subsection, we prove \tref{;tmain1}.

\begin{proof}[Proof of \tref{;tmain1}]
For each fixed $c>0$, by the scaling property of Brownian motion, the process 
$\bc =\{ \bc _{s}\} _{s\ge 0}$ defined by 
\begin{align*}
 \bc _{s}:=cB_{s/c^{2}}
\end{align*}
is also a Brownian motion. Then, by identity~\eqref{;invw} applied to $\bc $ with 
duration $c^{2}t$, we see that the process 
\begin{align}\label{;idenscl}
 B_{s/c^{2}}-\frac{1}{c}\log \left\{ 
 1+\frac{A_{s}(\bc )}{A_{c^{2}t}(\bc )}\left( 
 e^{2\bc _{c^{2}t}}-1
 \right) 
 \right\} ,\quad 0\le s\le c^{2}t,
\end{align}
is identical in law with $\{ B_{s/c^{2}}\} _{0\le s\le c^{2}t}$.
Note that a change of the variables yields   
\begin{align*}
 A_{s}(\bc )&=\int _{0}^{s}e^{2cB_{u/c^{2}}}\,du\\
 &=c^{2}A_{s/c^{2}}(cB)
\end{align*}
for every $0\le s\le c^{2}t$. Therefore, by the definition~\eqref{;defctc}  
of $\ctc $, the process in \eqref{;idenscl} is expressed as 
\begin{align*}
 \ct ^{(c)}(B)(s/c^{2}),\quad 0\le s\le c^{2}t,
\end{align*}
which proves the theorem.
\end{proof}

\subsection{Proof of \tref{;tmain2}}\label{;ssprf2}

This subsection is devoted to the proof of \tref{;tmain2}. Since the case 
of the transformation $\cg $, although handled just after 
the statement of \tref{;tmain2}, can also be discussed 
in the same way as that of $\cm $, we give proofs of both cases 
below. For each $c>0$, by the definition of $\ctc $, together with 
that of $\ct $ (recall \eqref{;defctc} and \eqref{;defct}, respectively), 
keep in mind the expression 
\begin{align}\label{;expctc}
 \ctc (\phi )(s)=\phi _{s}-\frac{1}{c}\log \left\{ 
 1+\frac{A_{s}(c\phi )}{A_{t}(c\phi )}\left( e^{2c\phi _{t}}-1\right) 
 \right\} ,\quad 0\le s\le t,
\end{align}
for every $\phi \in C([0,t];\R )$. We begin with 

\begin{lem}\label{;lconvs}
 For every $\phi \in C([0,t];\R )$, we have the following:
 \begin{itemize}
  \item[\thetag{i}] it holds that, for all $s\in [0,t]$,
  \begin{align}\label{;convg}
   \ctc(\phi )(s)\xrightarrow[c\downarrow 0]{}\cg (\phi )(s);
  \end{align}
  
  \item[\thetag{ii}] if, in addition, $\phi _{0}=0$, then it holds that, for all $s\in [0,t]$,
  \begin{align}\label{;convm}
   \ctc (\phi )(s)\xrightarrow[c\to \infty ]{}\cm (\phi )(s).
  \end{align}
 \end{itemize}
 Moreover, the above two convergences~\eqref{;convg} and \eqref{;convm} 
 are uniform.
\end{lem}

\begin{rem}\label{;rlconvs}
 Since, for every $c>0$, the function $\ctc (\phi )(s)-\phi _{s},\,0\le s\le t$, 
 is monotone (in fact, it is nonincreasing when $\phi _{t}\ge 0$ and nondecreasing 
 when $\phi _{t}\le 0$), the uniformity stated in the above lemma follows from 
 the fact that both of the limit functions 
 \begin{align}\label{;tftns}
  \cg (\phi )(s)-\phi _{s} \quad \text{and} \quad 
  \cm (\phi )(s)-\phi _{s},\quad 0\le s\le t, 
 \end{align}
 are continuous; see, e.g., \cite[pp.\ 81 and 270, Problem~127]{ps}, a fact 
 which is often referred to as Dini's second theorem. We append its statement 
 and proof in the appendix; see \lref{;tdini}. The graph of the former function in 
 \eqref{;tftns} connects the two points $(0,0)$ and $(t,-2\phi _{t})$ monotonically, 
 and, provided that $\phi _{0}=0$, so does that of the latter in view of 
 \eqref{;bvs} below. 
\end{rem}

\begin{proof}[Proof of \lref{;lconvs}]
In view of \eqref{;expctc}, assertion~\thetag{i} is immediate from the fact that 
$(1/x)\log (1+x)\xrightarrow[x\to 0]{}1$. As for assertion~\thetag{ii}, first 
observe that, when $\phi _{0}=0$, 
\begin{align}\label{;bvs}
 \ctc (\phi )(0)=\cm (\phi )(0)=0 \quad \text{and} \quad 
 \ctc (\phi )(t)=\cm (\phi )(t)=-\phi _{t}
\end{align}
for any $c>0$. Indeed, it is readily seen from \eqref{;expctc} that 
\begin{align}\label{;bvctc}
 \ctc (\phi )(0)=\phi _{0}=0 \quad \text{and} \quad 
 \ctc (\phi )(t)=-\phi _{t}.
\end{align}
On the other hand, by the definition~\eqref{;defcm} of $\cm $, 
\begin{align*}
 \cm (\phi )(0)&=\phi _{0}-\phi _{t}-\Bigl| 
 \phi _{t}+\phi _{0}-\max _{0\le u\le t}\phi _{u}
 \Bigr| +\Bigl| 
 \phi _{0}-\max _{0\le u\le t}\phi _{u}
 \Bigr| \\
 &=-\phi _{t}+\phi _{t}-\max _{0\le u\le t}\phi _{u}+\max _{0\le u\le t}\phi _{u}\\
 &=0,
\intertext{and}
 \cm (\phi )(t)&=\phi _{t}-\phi _{t}-\Bigl| 
 \phi _{t}+\max _{0\le u\le t}\phi _{u}-\phi _{t}
 \Bigr| +\Bigl| 
 \max _{0\le u\le t}\phi _{u}-\phi _{t}
 \Bigr| \\
 &=-\max _{0\le u\le t}\phi _{u}+\max _{0\le u\le t}\phi _{u}-\phi _{t}\\
 &=-\phi _{t},
\end{align*}
where, for the second line from the bottom, we have used 
$\max _{0\le u\le t}\phi _{u}\ge \phi _{0}=0$. 
Turning to the case $s\in (0,t)$, we rewrite, in expression~\eqref{;expctc}, 
\begin{align*}
 1+\frac{A_{s}(c\phi )}{A_{t}(c\phi )}\left( e^{2c\phi _{t}}-1\right) 
 &=\frac{1}{A_{t}(c\phi )}\left( 
 e^{2c\phi _{t}}\!\int _{0}^{s}e^{2c\phi _{u}}\,du+\int _{s}^{t}e^{2c\phi _{u}}\,du
 \right) ,
\end{align*}
and note that, as $c\to \infty $, 
\begin{align*}
 \frac{1}{c}\log \left( 
 e^{2c\phi _{t}}\!\int _{0}^{s}e^{2c\phi _{u}}\,du
 \right) \to 2\Bigl( 
 \phi _{t}+\max _{0\le u\le s}\phi _{u}
 \Bigr) , && 
 \frac{1}{c}\log 
 \int _{s}^{t}e^{2c\phi _{u}}\,du
 \to 2\max _{s\le u\le t}\phi _{u},
\end{align*}
as well as $(1/c)\log A_{t}(c\phi )\to 2\max _{0\le u\le t}\phi _{u}$, 
which may be seen from the general fact that, given a measure space, 
the $L^{p}$-norm ($p\ge 1$) of a measurable function defined on the space 
converges to its $L^{\infty }$-norm as $p\to \infty $ (see, e.g., \cite[p.~134, Problem~13.21]{sch}).
Combining these with \lref{;lfund}, we see that $\ctc (\phi )(s)$ 
converges to   
\begin{align*}
 \phi _{s}-2\max \Bigl\{ \phi _{t}+\max _{0\le u\le s}\phi _{u},\max _{s\le u\le t}\phi _{u}
 \Bigr\} +2\max _{0\le u\le t}\phi _{u}
\end{align*}
as $c\to \infty $. The last expression agrees with $\cm (\phi )(s)$ by noting the relations 
\begin{align*}
 2\max \{ a,b\} =a+b+|a-b|\ (a,b\in \R ) && \text{and} && 
 \max _{0\le u\le t}\phi _{u}
 =\max \Bigl\{ 
 \max _{0\le u\le s}\phi _{u},\max _{s\le u\le t}\phi _{u}
 \Bigr\} . 
\end{align*}
The proof of assertion~\thetag{ii} is completed. Finally, the fact that 
both of the two convergences~\eqref{;convg} and \eqref{;convm} are 
uniform is a consequence of \lref{;tdini} as noted in \rref{;rlconvs}. 
\end{proof}

We are in a position to prove \tref{;tmain2}.

\begin{proof}[Proof of \tref{;tmain2}]
Since both $\{ \cg (B)(s)\} _{0\le s\le t}$ and $\{ \cm (B)(s)\} _{0\le s\le t}$ 
are continuous processes, it suffices to show that their finite-dimensional 
distributions agree with those of $B$ up to time $t$. To this end, pick 
$n\in \N $, $0\le s_{1}<\cdots <s_{n}\le t$ and a bounded continuous 
function $f:\R ^{n}\to \R $ arbitrarily. Then, by \tref{;tmain1}, we have 
\begin{align*}
 \ex \!\left[ 
 f\bigl( 
 \ctc (B)(s_{1}),\ldots ,\ctc (B)(s_{n})
 \bigr) 
 \right] 
 =\ex \!\left[ 
 f(B_{s_{1}},\ldots ,B_{s_{n}})
 \right] 
\end{align*}
for any $c>0$. Thanks to the boundedness and continuity of $f$, the 
bounded convergence theorem entails that, as $c\downarrow 0$, the 
left-hand side converges to 
\begin{align*}
 \ex \!\left[ 
 f\bigl( 
 \cg (B)(s_{1}),\ldots ,\cg (B)(s_{n})
 \bigr) 
 \right] 
\end{align*}
by \lref{;lconvs}\thetag{i}, while as $c\to \infty $, it converges to 
\begin{align*}
 \ex \!\left[ 
 f\bigl( 
 \cm (B)(s_{1}),\ldots ,\cm (B)(s_{n})
 \bigr) 
 \right] 
\end{align*}
by \lref{;lconvs}\thetag{ii}. This concludes the proof of the theorem.
\end{proof}

\begin{rem}\label{;rmin}
By virtue of the symmetry of Brownian motion, the process given by 
$-\cm (-B)$, namely 
\begin{align*}
 B_{s}-B_{t}+\Bigl| 
 B_{t}+\min _{0\le u\le s}B_{u}-\min _{s\le u\le t}B_{u}
 \Bigr| 
 -\Bigl| 
 \min _{0\le u\le s}B_{u}-\min _{s\le u\le t}B_{u}
 \Bigr| ,\quad 0\le s\le t,
\end{align*}
is also a Brownian motion.
\end{rem}

\section{Properties of $\ctc $, $\cg $ and $\cm $}\label{;sptranss}

As was already referred to in \sref{;intro},
the path transformations $\ctc ,\,c>0$, $\cg $ and $\cm $ have 
two properties in common: one is that they are an involution and 
the other is the commutativity with time reversal. The aim of 
this section is to investigate properties of these transformations 
including the above-mentioned two, which we think are of interest in their 
own right; in particular, we reveal that the transformation $\cm $ 
preserves Pitman's transformation and see how \cref{;ctmain2} 
is deduced.

Following the notation in \cite{dmy} by Donati-Martin, Matsumoto and Yor, 
define the transformation $Z:C([0,\infty );\R )\to C([0,\infty );\R )$ by
\begin{align}\label{;defz}
 Z_{s}(\phi ):=e^{-\phi _{s}}A_{s}(\phi ),\quad s\ge 0,\ \phi \in C([0,\infty );\R ),
\end{align}
which admits the following relationship with the transformation $A$: 
\begin{align}\label{;deria}
 \frac{d}{ds}\frac{1}{A_{s}(\phi )}=-\frac{1}{\left\{ 
 Z_{s}(\phi )
 \right\} ^{2}},\quad s>0.
\end{align}
With $t>0$ kept fixed, as in \cite{har22, har22+}, we also denote by $R$ the time-reversal 
operator defined by
\begin{align}\label{;trev}
 R(\phi )(s):=\phi _{t-s}-\phi _{t},\quad 0\le s\le t,
 \ \phi \in C([0,t];\R ),
\end{align}
under which we know the invariance of the law of $B$ up to time $t$: 
\begin{align}\label{;itrev}
 \left\{ R(B)(s)\right\} _{0\le s\le t}
 \eqd 
 \{ B_{s}\} _{0\le s\le t}.
\end{align}
(Be aware that the usual time reversal of Brownian motion 
refers to $\{ -R (B)(s)\} _{0\le s\le t}$ in our notation.) 
In what follows, we denote by $\id $ the identity map on 
$C([0,t];\R )$ and restrict the transformation $Z$ to $C([0,t];\R )$. 
In the case $c=1$, namely the case of $\ct $, the following properties of 
$\ctc $ are proven in \cite[Proposition~2.1]{har22+}.

\begin{prop}\label{;ppropctc}
For every $c>0$, we have the following:
\begin{itemize}
 \item[\thetag{i}] $Z\bigl( c\ctc (\phi )\bigr) =Z(c\phi )$ for any $\phi \in C([0,t];\R )$;
 
 \item[\thetag{ii}] $\ctc \circ \ctc =\id $;
 
 \item[\thetag{iii}] for any $\phi \in C([0,t];\R )$ with $\phi _{0}=0$,
 \begin{align*}
   \bigl( R\circ \ctc \bigr) (\phi )=\bigl( \ctc \circ R\bigr) (\phi ).
 \end{align*}
\end{itemize}
\end{prop}

These properties of each $\ctc $ are immediate from those of $\ct $ 
investigated in \cite{har22+}, which, for the reader's convenience, we 
put in \lref{;la1} in the appendix with proof.

\begin{proof}[Proof of \pref{;ppropctc}]
\thetag{i} By the definition~\eqref{;defctc} of $\ctc $ and by \lref{;la1}\thetag{i},
\begin{align*}
 Z\bigl( c\ctc (\phi )\bigr) &=Z(\ct (c\phi ))\\
 &=Z(c\phi ).
\end{align*}

\thetag{ii} For any $\phi \in C([0,t];\R )$, we have, by the definition~\eqref{;defctc} 
of $\ctc $,
\begin{align*}
 \ctc \bigl( 
 \ctc (\phi )
 \bigr) &=\frac{1}{c}\ct (\ct (c\phi ))\\
 &=\frac{1}{c}\cdot c\phi \\
 &=\phi ,
\end{align*}
where we used \lref{;la1}\thetag{ii} for the second line.

\thetag{iii} Let $\phi \in C([0,t];\R )$ be such that $\phi _{0}=0$. 
By the definition~\eqref{;defctc} of $\ctc $, since $R $ is a 
linear transformation, 
\begin{align*}
 R\bigl( \ctc (\phi )\bigr) &=\frac{1}{c}R(\ct (c\phi ))\\
 &=\frac{1}{c}\ct (R(c\phi ))\\
 &=\frac{1}{c}\ct (cR(\phi )),
\end{align*}
which is $\ctc \bigl( R(\phi )\bigr) $ by definition. Here we used 
\lref{;la1}\thetag{iii} for the second line and the linearity of $R$ 
again for the third.
\end{proof}

We turn to the transformation $\cg $.

\begin{prop}\label{;ppropcg}
The following hold: 
\begin{itemize}
 \item[\thetag{i}] $\cg \circ \cg =\id $;
 
 \item[\thetag{ii}] for any $\phi \in C([0,t];\R )$ with $\phi _{0}=0$,
 \begin{align*}
  R(\cg (\phi ))=\cg (R(\phi )).
 \end{align*}
\end{itemize}
\end{prop}

We may infer those properties of $\cg $ from \pref{;ppropctc}, 
but direct proofs are of course possible as seen below.

\begin{proof}[Proof of \pref{;ppropcg}]
\thetag{i} For every $\phi \in C([0,t];\R )$ and $s\in [0,t]$, we have, by the 
definition~\eqref{;defcg} of $\cg $, 
\begin{align*}
 \cg (\cg (\phi ))(s)&=\cg (\phi )(s)-\frac{2s}{t}\cg (\phi )(t)\\
 &=\phi _{s}-\frac{2s}{t}\phi _{t}-\frac{2s}{t}\cdot (-\phi _{t})\\
 &=\phi _{s}.
\end{align*}

\thetag{ii} By the definition~\eqref{;trev} of $R$, for every $s\in [0,t]$,
\begin{align*}
 R(\cg (\phi ))(s)&=\cg (\phi )(t-s)-\cg (\phi )(t)\\
 &=\phi _{t-s}-\frac{2(t-s)}{t}\phi _{t}+\phi _{t},
\end{align*}
which, when $\phi _{0}=0$, is equal to 
\begin{align*}
 \phi _{t-s}-\phi _{t}-\frac{2s}{t}(\phi _{0}-\phi _{t})
 &=R(\phi )(s)-\frac{2s}{t}R(\phi )(t)\\
 &=\cg (R(\phi ))(s)
\end{align*}
as claimed.
\end{proof}

\sloppy 
Regarding Pitman's transformation $\cp $ defined in \eqref{;defcp} 
as a transformation on $C([0,t];\R )$, we have 

\begin{prop}\label{;ppropcm}
Restricted to the space of $\phi $'s in $C([0,t];\R )$ satisfying $\phi _{0}=0$, 
the transformation $\cm $ enjoys the following properties:
\begin{align*}
 \thetag{i}\ \cp \circ \cm =\cp ; && 
 \thetag{ii}\ \cm \circ \cm =\id ; && 
 \thetag{iii}\ R\circ \cm =\cm \circ R.
\end{align*}
\end{prop}

\begin{proof}
We fix $\phi \in C([0,t];\R )$ and $s\in [0,t]$ arbitrarily below. 
By \pref{;ppropctc}\thetag{i}, we have
\begin{align}\label{;a}
 Z_{s}\bigl( c\ctc (\phi )\bigr) =Z_{s}(c\phi ).
\end{align}
Let $s$ be positive for the time being. As for the right-hand side of \eqref{;a}, 
in view of the definition~\eqref{;defz} of the transformation $Z$, 
\begin{equation}\label{;b}
 \begin{split}
  \frac{1}{c}\log Z_{s}(c\phi )&=\frac{1}{c}\log A_{s}(c\phi )-\phi _{s}\\
  &\xrightarrow[c\to \infty ]{}\cp (\phi )(s).
 \end{split}
\end{equation}
On the other hand, as for the left-hand side, pick $\ve >0$ arbitrarily. 
Then, assuming that $\phi _{0}=0$, we have, for every sufficiently large 
$c$, 
\begin{align*}
 \max _{0\le u\le t}\bigl| 
 \ctc (\phi )(u)-\cm (\phi )(u)
 \bigr| <\ve 
\end{align*}
since the pointwise convergence~\eqref{;convm} of $\ctc (\phi )$ to 
$\cm (\phi )$ is uniform, whence, by the 
definition~\eqref{;defz} of $Z$,
\begin{align*}
 e^{-2c\ve }e^{-c\ctc (\phi )(s)}A_{s}(c\cm (\phi ))
 \le Z_{s}\bigl( c\ctc (\phi )\bigr) \le 
 e^{2c\ve }e^{-c\ctc (\phi )(s)}A_{s}(c\cm (\phi )).
\end{align*}
It then follows that 
\begin{align*}
 -2\ve +\cp (\cm (\phi ))(s)&\le 
 \liminf _{c\to \infty }\frac{1}{c}\log Z_{s}\bigl( c\ctc (\phi )\bigr) \\
 &\le \limsup _{c\to \infty }\frac{1}{c}\log Z_{s}\bigl( c\ctc (\phi )\bigr) 
 \le 2\ve +\cp (\cm (\phi ))(s).
\end{align*}
Therefore, thanks to the arbitrariness of $\ve $, 
\begin{align*}
 \lim _{c\to \infty }\frac{1}{c}\log Z_{s}\bigl( c\ctc (\phi )\bigr) 
 =\cp (\cm (\phi ))(s),
\end{align*}
which, together with \eqref{;a} and \eqref{;b}, entails that 
\begin{align*}
 \cp (\cm (\phi ))(s)=\cp (\phi )(s)
\end{align*}
for any $0<s\le t$. Since both sides of the last relation are continuous 
at the origin, we have \thetag{i}.

We proceed to the proof of \thetag{iii} next, which we will do 
by a direct computation instead of appealing to 
\pref{;ppropctc}\thetag{iii}. By the definition~\eqref{;trev} of $R$, 
\begin{align*}
 R(\cm (\phi ))(s)&=\cm (\phi )(t-s)-\cm (\phi )(t)\\
 &=\cm (\phi )(t-s)+\phi _{t},
\end{align*}
where, supposing that $\phi _{0}=0$, the second line is due to 
\eqref{;bvs}. On the other hand, by the definition~\eqref{;defcm} 
of $\cm $, observe that 
\begin{align*}
 &\cm (R(\phi ))(s)\\
 &=R(\phi )(s)-R(\phi )(t)
 -\Bigl| 
 R(\phi )(t)+\max _{0\le u\le s}R(\phi )(u)-\max _{s\le u\le t}R(\phi )(u)
 \Bigr| \\
 &\quad +\Bigl| 
 \max _{0\le u\le s}R(\phi )(u)-\max _{s\le u\le t}R(\phi )(u)
 \Bigr| \\
 &=\phi _{t-s}-\phi _{t}-(\phi _{0}-\phi _{t})
 -\Bigl| 
 \phi _{0}-\phi _{t}+\max _{t-s\le u\le t}\phi _{u}
 -\max _{0\le u\le t-s}\phi _{u}
 \Bigr| \\
 &\quad +\Bigl| 
 \max _{t-s\le u\le t}\phi _{u}
 -\max _{0\le u\le t-s}\phi _{u}
 \Bigr| ,
\end{align*}
which is also $\cm (\phi )(t-s)+\phi _{t}$ when $\phi _{0}=0$.

Turning to the proof of \thetag{ii}, by the definition~\eqref{;defcm} of $\cm $, 
we have 
\begin{equation}\label{;c}
 \begin{split}
 &\cm (\cm (\phi ))(s)\\
 &=\cm (\phi )(s)-\cm (\phi )(t)
 -\Bigl| 
 \cm (\phi )(t)+\max _{0\le u\le s}\cm (\phi )(u)-\max _{s\le u\le t}\cm (\phi )(u)
 \Bigr| \\
 &\quad +\Bigl| 
 \max _{0\le u\le s}\cm (\phi )(u)-\max _{s\le u\le t}\cm (\phi )(u)
 \Bigr| .
\end{split}
\end{equation}
Since $\cm (\phi )(t)=-\phi _{t}$ when $\phi _{0}=0$, and 
\begin{align*}
 \max _{0\le u\le s}\cm (\phi )(u)-\max _{s\le u\le t}\cm (\phi )(u)
 =\max _{0\le u\le s}\phi _{u}-\max _{s\le u\le t}\phi _{u}+\phi _{t}
\end{align*}
by \thetag{i} and by \lref{;lrevp} below, we insert these into \eqref{;c} 
to get 
\begin{align*}
 &\cm (\cm (\phi ))(s)\\
 &=\cm (\phi )(s)+\phi _{t}-\Bigl| 
 -\phi _{t}+\max _{0\le u\le s}\phi _{u}-\max _{s\le u\le t}\phi _{u}+\phi _{t}
 \Bigr| 
 +\Bigl| 
 \max _{0\le u\le s}\phi _{u}-\max _{s\le u\le t}\phi _{u}+\phi _{t}
 \Bigr| \\
 &=\cm (\phi )(s)+\phi _{s}-\cm (\phi )(s)\\
 &=\phi _{s}
\end{align*}
as expected.
\end{proof}

The proof of \pref{;ppropcm} is completed once we establish the 
following lemma, which may be regarded as a complement to 
property~\thetag{i} in the proposition.

\begin{lem}\label{;lrevp}
It holds that, for every $\phi \in C([0,t];\R )$ with $\phi _{0}=0$,
\begin{align*}
 2\max _{s\le u\le t}\cm (\phi )(u)-\cm (\phi )(s)
 =2\max _{s\le u\le t}\phi _{u}-\phi _{s}-2\phi _{t},\quad 
 0\le s\le t.
\end{align*}
\end{lem}

\begin{proof}
For every $s\in [0,t]$, by property~\thetag{iii} in \pref{;ppropcm}, 
\begin{align*}
 \max _{0\le u\le t-s}R(\cm (\phi ))(u)
 =\max _{0\le u\le t-s}\cm (R(\phi ))(u).
\end{align*}
The left-hand side is equal to 
\begin{align*}
 \max _{0\le u\le t-s}\cm (\phi )(t-u)-\cm (\phi )(t)
 =\max _{s\le u\le t}\cm (\phi )(u)+\phi _{t},
\end{align*}
whereas the right-hand side is equal, by 
property~\thetag{i} in \pref{;ppropcm}, to 
\begin{align*}
 &\max _{0\le u\le t-s}R(\phi )(u)-\frac{1}{2}R(\phi )(t-s)+\frac{1}{2}\cm (R(\phi ))(t-s)\\
 &=\max _{0\le u\le t-s}(\phi _{t-u}-\phi _{t})-\frac{1}{2}(\phi _{s}-\phi _{t})
 +\frac{1}{2}\{ \cm (\phi )(s)+\phi _{t}\} \\
 &=\max _{s\le u\le t}\phi _{u}-\frac{1}{2}\phi _{s}+\frac{1}{2}\cm (\phi )(s),
\end{align*}
hence the claim. Here, for the second line, we used 
$\cm (R(\phi ))(t-s)=R(\cm (\phi ))(t-s)$ by \pref{;ppropcm}\thetag{iii} and 
$\cm (\phi )(t)=-\phi _{t}$ as in \eqref{;bvs}.
\end{proof}

\pref{;ppropcm} enables us to prove \cref{;ctmain2}.

\begin{proof}[Proof of \cref{;ctmain2}]
Given a nonnegative measurable function $F$ on $C([0,t];\R )$, 
consider a function $G$ of the form 
\begin{align*}
 G(\phi )=F(\phi )e^{-\mu \phi _{t}-\mu ^{2}t/2},\quad 
 \phi \in C([0,t];\R ).
\end{align*}
By \tref{;tmain2}, we have 
\begin{align*}
 \ex \!\left[ 
 G(\cm (B))
 \right] =\ex \!\left[ 
 G(B)
 \right] .
\end{align*}
By the fact that $\cm (B)(t)=-B_{t}$ in view of \eqref{;bvs}, the 
Cameron--Martin formula entails that 
\begin{align*}
 \ex \!\left[ 
 F\bigl( 
 \cm (\db{\mu })
 \bigr) 
 \right] =\ex \!\left[ 
 F\bigl( \db{-\mu }\bigr) 
 \right] .
\end{align*}
Therefore, thanks to the arbitrariness of $F$, 
\begin{align*}
 \bigl\{ \cm (\db{\mu })(s)\bigr\} _{0\le s\le t}
 \eqd \bigl\{ 
 \db{-\mu }_{s}
 \bigr\} _{0\le s\le t}.
\end{align*}
Hence we have the joint identity in law 
\begin{align*}
 &\left\{ 
 \bigl( 
 (\cm \circ \cm )(\db{\mu })(s),\cm (\db{\mu })(s),(\cp \circ \cm )(\db{\mu })(s)
 \bigr) 
 \right\} _{0\le s\le t}\\
 &\eqd 
 \left\{ 
 \bigl( 
 \cm (\db{-\mu })(s),\db{-\mu }_{s},\cp (\db{-\mu })(s)
 \bigr) 
 \right\} _{0\le s\le t},
\end{align*}
which is the claim since the left-hand side agrees with that of 
\eqref{;jointl} by \thetag{i} and \thetag{ii} of \pref{;ppropcm}.
\end{proof}

Let $\T $ be any of the transformations $\ctc $ (with any $c>0$), $\cg $ and $\cm $. 
Notice that $\T (B)(t)=-B_{t}$. We conclude this section with a remark on the 
commutativity of $\T $ with the time-reversal operator $R$ as observed in 
\psref{;ppropctc}--\ref{;ppropcm}. 

\begin{rem}\label{;rcons}
By the same argument as in the above proof, we have, for any $\mu \in \R $,
\begin{align}\label{;oppd}
 \bigl\{ \T (\db{\mu })(s)\bigr\} _{0\le s\le t}
 \eqd \bigl\{ 
 \db{-\mu }_{s}
 \bigr\} _{0\le s\le t},
\end{align}
which is consistent with the commutativity with $R$ 
as seen below: 
\begin{align*}
 \bigl\{ 
 (R\circ \T )(\db{\mu })(s)
 \bigr\} _{0\le s\le t}
 &\eqd \bigl\{ 
 R(\db{-\mu })(s)
 \bigr\} _{0\le s\le t}\\
 &\eqd \bigl\{ 
 \db{\mu }_{s}
 \bigr\} _{0\le s\le t}\\
 &\eqd \bigl\{ 
 \T (\db{-\mu })(s)
 \bigr\} _{0\le s\le t}\\
 &\eqd \bigl\{ 
 (\T \circ R)(\db{\mu })(s)
 \bigr\} _{0\le s\le t},
\end{align*}
where the first and third lines are due to \eqref{;oppd} and 
the second and fourth to \eqref{;itrev}; in fact, because of 
the commutativity, the leftmost and rightmost sides are identical.
\end{rem}

\section{Proof of \pref{;pexpress}}\label{;sppexpress}

In this section, we prove \pref{;pexpress}. We also reveal how the invariance 
in law of Brownian motion under the transformation $\cm $ is connected 
with Pitman's $2M-X$ theorem via the proposition. 

We begin with the following lemma: 
\begin{lem}\label{;lexpress}
For every $\phi \in C([0,t];\R )$ and $c>0$, it holds that 
\begin{align}
 \phi _{s}&=-\frac{1}{c}\log\left\{ 
 Z_{s}(c\phi )\int _{s}^{t}\frac{du}{\left( 
 Z_{u}(c\phi )
 \right) ^{2}}+\frac{Z_{s}(c\phi )}{Z_{t}(c\phi )}e^{-c\phi _{t}}
 \right\} ,\label{;iipd}\\
 \ctc (\phi )(s)&=-\frac{1}{c}\log\left\{ 
 Z_{s}(c\phi )\int _{s}^{t}\frac{du}{\left( 
 Z_{u}(c\phi )
 \right) ^{2}}+\frac{Z_{s}(c\phi )}{Z_{t}(c\phi )}e^{c\phi _{t}}
 \right\} ,\label{;impd}
\end{align}
for all $0<s\le t$.
\end{lem}

\begin{proof}
In view of \eqref{;deria}, we have, for every $0<s\le t$,  
\begin{align*}
 \int _{s}^{t}\frac{du}{\left\{ 
 Z_{u}(c\phi )
 \right\} ^{2}}=\frac{1}{A_{s}(c\phi )}-\frac{1}{A_{t}(c\phi )},
\end{align*}
and hence, by the definition~\eqref{;defz} of the transformation $Z$, 
\begin{align*}
 Z_{s}(c\phi )\int _{s}^{t}\frac{du}{\left\{ 
 Z_{u}(c\phi )
 \right\} ^{2}}=e^{-c\phi _{s}}-\frac{Z_{s}(c\phi )}{Z_{t}(c\phi )}e^{-c\phi _{t}},
\end{align*}
from which relation~\eqref{;iipd} follows. We apply \eqref{;iipd} to 
$\ctc (\phi )\in C([0,t];\R )$ to obtain \eqref{;impd} using 
\pref{;ppropctc}\thetag{i} and the fact that $\ctc (\phi )(t)=-\phi _{t}$ 
as in \eqref{;bvctc}.
\end{proof}

\begin{rem}
Both of the right-hand sides of \eqref{;iipd} and \eqref{;impd} 
converge to $\phi _{0}$ as $s\downarrow 0$ by the continuity of 
$\phi $ and $\ctc (\phi )$; recall that, as already seen in \eqref{;bvctc}, 
$\ctc (\phi )(0)=\phi _{0}$ for any $c>0$. 
\end{rem}

Using \lref{;lexpress}, we prove \pref{;pexpress}.

\begin{proof}[Proof of \pref{;pexpress}]
Given $\phi \in C([0,t];\R )$, let $\cs ^{1}(\phi )(s)$ (resp.\ $\cs ^{2}(\phi )(s)$) 
denote the right-hand side of \eqref{;iip} (resp.\ \eqref{;imp}) for every $0\le s\le t$. 
Because $\cs ^{1}(\phi )$ and $\cs ^{2}(\phi )$, as well as $\cm (\phi )$, 
are continuous functions on $[0,t]$, it suffices to 
prove the asserted relations 
\begin{align}\label{;arel}
 \phi _{s}=\cs ^{1}(\phi )(s), && \cm (\phi )(s)=\cs ^{2}(\phi )(s),
\end{align}
for $0<s<t$.

Fix $s\in (0,t)$ below. First we verify that, as $c\to \infty $, 
\begin{align}\label{;convmp}
 \frac{1}{c}\log \int _{s}^{t}\frac{du}{\left\{ 
 Z_{u}(c\phi )
 \right\} ^{2}}
 \to -2\min _{s\le u\le t}\cp (\phi )(u).
\end{align}
To this end, note that, as was observed in \eqref{;b}, 
\begin{align*}
 \frac{1}{c}\log Z_{u}(c\phi )\xrightarrow[c\to \infty ]{}\cp (\phi )(u)
\end{align*}
for every $u\in [s,t]$; moreover, this convergence is uniform 
thanks to \lref{;tdini} because, for each $c>0$, the function 
\begin{align*}
 \frac{1}{c}\log Z_{u}(c\phi )+\phi _{u}=\frac{1}{c}\log A_{u}(c\phi ),\quad 
 s\le u\le t,
\end{align*}
is nondecreasing and its limit function $\cp (\phi )(u)+\phi _{u},\,s\le u\le t$, 
is continuous. Fix $\ve >0$ arbitrarily and take $c$ so large that 
\begin{align*}
 \max _{s\le u\le t}\left| 
 \frac{1}{c}\log Z_{u}(c\phi )-\cp (\phi )(u)
 \right| <\ve .
\end{align*}
Writing 
$
1/\left\{ Z_{u}(c\phi )\right\} ^{2}
=\exp \left\{ 
-2c\cdot (1/c)\log Z_{u}(c\phi )
\right\} 
$, we then have the bounds 
\begin{align*}
 -2\ve +\frac{1}{c}\log \int _{s}^{t}\exp \left\{ 
 -2c\cp (\phi )(u)
 \right\} du 
 &<\frac{1}{c}\log \int _{s}^{t}\frac{du}{\left\{ 
 Z_{u}(c\phi )
 \right\} ^{2}}\\
 &<2\ve +\frac{1}{c}\log \int _{s}^{t}\exp \left\{ 
 -2c\cp (\phi )(u)
 \right\} du.
\end{align*}
Since 
\begin{align*}
 \frac{1}{c}\log \int _{s}^{t}\exp \left\{ 
 -2c\cp (\phi )(u)
 \right\} du \xrightarrow[c\to \infty ]{}-2\min _{s\le u\le t}\cp (\phi )(u)
\end{align*}
(see, e.g., \cite[p.~134, Problem~13.21]{sch}), we obtain \eqref{;convmp} by the 
same reasoning as in the proof of \pref{;ppropcm}\thetag{i}.

By relation~\eqref{;iipd}, 
\begin{align*}
 \phi _{s}=-\frac{1}{c}\log Z_{s}(c\phi )
 -\frac{1}{c}\log \left\{ 
 \int _{s}^{t}\frac{du}{\left( 
 Z_{u}(c\phi )
 \right) ^{2}}+\frac{e^{-c\phi _{t}}}{Z_{t}(c\phi )}
 \right\} ,
\end{align*}
which, by \eqref{;convmp} and \lref{;lfund}, converges to 
\begin{align*}
 &-\cp (\phi )(s)-\max \Bigl\{ 
 -2\min _{s\le u\le t}\cp (\phi )(u),-\phi _{t}-\cp (\phi )(t)
 \Bigr\} 
\end{align*}
as $c\to \infty $. The last expression agrees with $\cs ^{1}(\phi )(s)$ 
and proves the former relation in \eqref{;arel}. 
On the other hand, relation~\eqref{;impd} entails, in the same way as above, that 
\begin{align*}
 \ctc (\phi )(s)\xrightarrow[c\to \infty ]{}\cs ^{2}(\phi )(s).
\end{align*} 
We have observed in the proof of \lref{;lconvs}\thetag{ii} that, 
regardless of whether $\phi _{0}=0$ or not, 
\begin{align*}
 \ctc (\phi )(s)\xrightarrow[c\to \infty ]{}\cm (\phi )(s)
\end{align*}
for all $s\in (0,t)$. Therefore we also obtain the latter relation in 
\eqref{;arel}.
\end{proof}

\begin{rem}\label{;rinvol}
Property~\thetag{ii} in \pref{;ppropcm} may also be proven by using 
\pref{;pexpress} in such a way that, given $\phi \in C([0,t];\R )$ with 
$\phi _{0}=0$ and for every $0\le s\le t$, thanks to \eqref{;imp} and 
\pref{;ppropcm}\thetag{i}, 
\begin{align*}
 (\cm \circ \cm )(\phi )(s)=-\cp (\phi )(s)+\min \Bigl\{ 
 2\min _{s\le u\le t}\cp (\phi )(u),\cp (\phi )(t)-\cm (\phi )(t)
 \Bigr\} ,
\end{align*}
which, in view of \eqref{;iip}, agrees with $\phi _{s}$ because 
$\cm (\phi )(t)=-\phi _{t}$ as in \eqref{;bvs}.
\end{rem}

For each $a\ge 0$, let $\bes ^{a}=\{ \bes ^{a}_{s}\} _{s\ge 0}$ denote a 
three-dimensional Bessel process starting from $a$. 
Pitman's $2M-X$ theorem \cite[Theorem~1.3]{jwp} asserts that 
\begin{align}\label{;pit}
 \Bigl\{ 
 \Bigl( \cp (B)(s),\,\max _{0\le u\le s}B_{u}\Bigr) 
 \Bigr\} _{s\ge 0}\eqd 
 \Bigl\{ 
 \Bigl( 
 \bes ^{0}_{s},\,\inf _{u\ge s}\bes ^{0}_{u}
 \Bigr) 
 \Bigr\} _{s\ge 0}
\end{align} 
(refer to \cite[Theorem~VI.3.5]{ry} for the above form of the statement). In the rest of the section, 
with the help of \pref{;pexpress}, we discuss how identity~\eqref{;pit} explains the fact that 
\begin{align}\label{;invm}
 \{ \cm (B)(s)\} _{0\le s\le t}\eqd 
  \{ B_{s}\} _{0\le s\le t}
\end{align}
as shown in \tref{;tmain2}. 

Observe from \eqref{;pit} the identity in law 
\begin{align}\label{;idenp}
 \left( 
 \{ \cp (B)(s)\} _{0\le s\le t},\,B_{t}
 \right) \eqd 
 \Bigl( 
 \bigl\{ \bes ^{0}_{s}\bigr\} _{0\le s\le t},\,2\inf _{u\ge t}\bes ^{0}_{u}-\bes ^{0}_{t}
 \Bigr) .
\end{align}
Let $U$ denote a random variable subject to the uniform distribution 
on $(0,1)$. Because of the fact that, given $a>0$, $\inf _{s\ge 0}\bes ^{a}_{s}$ 
is distributed as $aU$ (see, e.g., \cite[Corollary~VI.3.4]{ry}), and by the 
Markov property of $\bes ^{0}$, the right-hand side of \eqref{;idenp} is 
identical in law with 
\begin{align*}
 \bigl( 
 \bigl\{ 
 \bes ^{0}_{s}
 \bigr\} _{0\le s\le t},\,(2U-1)\bes ^{0}_{t}
 \bigr) ,
\end{align*}
where $U$ is assumed to be independent of $\bes ^{0}$. Since 
$2U-1\eqd 1-2U$ (in fact, the random variable $2U-1$ is 
uniformly distributed on $(-1,1)$), we arrive at the symmetry of 
the law of $B_{t}$ given $\cp (B)$ up to time $t$: 
\begin{align}\label{;csym}
 \left( 
 \{ \cp (B)(s)\} _{0\le s\le t},\,B_{t}
 \right) 
 \eqd  \left( 
 \{ \cp (B)(s)\} _{0\le s\le t},\,-B_{t}
 \right) .
\end{align}
In particular, we have 
\begin{align*}
 &\Bigl\{ 
 -\cp (B)(s)+\min \Bigl\{ 
 2\min _{s\le u\le t}\cp (B)(u),\cp (B)(t)-B_{t}
 \Bigr\} 
 \Bigr\} _{0\le s\le t}\\
 &\eqd \Bigl\{ 
 -\cp (B)(s)+\min \Bigl\{ 
 2\min _{s\le u\le t}\cp (B)(u),\cp (B)(t)+B_{t}
 \Bigr\}  
 \Bigr\} _{0\le s\le t},
\end{align*}
which is \eqref{;invm} in view of \pref{;pexpress}.

\begin{rem}\label{;rpinv}
\thetag{1} If we denote by $\cm _{x}$ the transformation from 
$C([0,t];\R )$ to itself defined by 
\begin{align*}
 &\cm _{x}(\phi )(s)\\
 &:=-\cp (\phi )(s)+\min \Bigl\{ 
 2\min _{s\le u\le t}\cp (\phi )(u),\cp (\phi )(t)+x
 \Bigr\} ,\quad 0\le s\le t,\ \phi \in C([0,t];\R ),
\end{align*}
for each $x\in \R $, then the above argument entails the following 
disintegration formula for the Wiener measure on $C([0,t];\R )$, 
which is of interest in its own right: 
\begin{align*}
 \ex [F(B)]
 &=\ex \!\left[ 
 \int _{-\cp (\phi )(t)}^{\cp (\phi )(t)}
 \frac{dx}{2\cp (\phi )(t)}\,F\bigl( 
 \cm _{x}(\phi )
 \bigr) \Bigg| _{\phi =B}
 \right] \\
 &=\ex \!\left[ 
 \int _{-1}^{1}\frac{dx}{2}\,
 F\bigl( 
 \cm _{\cp (\phi )(t)x}(\phi )
 \bigr) \bigg| _{\phi =B}
 \right] 
\end{align*}
for every nonnegative measurable function $F$ on $C([0,t];\R )$.  
Here the second equality is due to a change of the variables. Moreover, 
if we consider the transformation 
\begin{align*}
\phi _{s}-\frac{1}{c}\log \left\{ 
 1+\frac{A_{s}(c\phi )}{A_{t}(c\phi )}\left( 
 e^{c\phi _{t}-cx}-1
 \right) 
 \right\} ,\quad 0\le s\le t,\ \phi \in C([0,t];\R ),
\end{align*}
then, by the same argument as in 
the proofs of \pref{;pexpress} and \lref{;lconvs}\thetag{ii}, each $\cm _{x}$ 
is seen to admit the following representation in terms of past and future maxima: 
\begin{align*}
 \cm _{x}(\phi )(s)
 =\phi _{s}-\frac{\phi _{t}-x}{2}-\left| 
 \frac{\phi _{t}-x}{2}+\max _{0\le u\le s}\phi _{u}-\max _{s\le u\le t}\phi _{u}
 \right| 
 +\Bigl| 
 \max _{0\le u\le s}\phi _{u}-\max _{s\le u\le t}\phi _{u}
 \Bigr| 
\end{align*}
for every $0\le s\le t$ and $\phi \in C([0,t];\R )$. We will return to the 
above observations elsewhere.

\thetag{2} In the same way as in the proof of \cref{;ctmain2}, we see 
from \eqref{;csym} that, for every $\mu \in \R $, 
\begin{align*}
 \bigl\{ 
 \cp (\db{\mu })(s) 
 \bigr\} _{0\le s\le t}
 \eqd  \bigl\{ 
 \cp (\db{-\mu })(s) 
 \bigr\} _{0\le s\le t},
\end{align*} 
which entails \eqref{;pinv} as $t>0$ is arbitrarily taken.
\end{rem}

\bigskip 
\noindent 
{\bf Acknowledgements.}  The author would like to thank the anonymous 
referees for their constructive comments.

\appendix 
\section*{Appendix}
\renewcommand{\thesection}{A}
\setcounter{equation}{0}
\setcounter{thm}{0}

Here,  based on \cite{har22+}, with providing their proofs, we list properties 
of $\ct $ recalled in \sref{;sptranss}. A proof of the identity~\eqref{;invw} 
in law is also given briefly. In addition, we state two fundamental lemmas 
used in the paper; although their proofs are standard, we also append them 
for the reader's convenience.

\subsection{Properties of $\ct $}

The following properties of the transformation $\ct $ are 
taken from \cite[Proposition~2.1]{har22+}.

\begin{lem}\label{;la1}
We have the following:
\begin{itemize}
 \item[\thetag{i}] $Z\circ \ct =Z$;
 
 \item[\thetag{ii}] $\ct \circ \ct =\id $;
 
 \item[\thetag{iii}] for any $\phi \in C([0,t];\R )$ with $\phi _{0}=0$,
 \begin{align*}
   (R\circ \ct )(\phi )=(\ct \circ R)(\phi ).
 \end{align*}
\end{itemize}
\end{lem}

\begin{proof}
\thetag{i} In view of the definition~\eqref{;defct} of $\ct $, 
a direct computation shows that, for every $\phi \in C([0,t];\R )$, 
\begin{align*}
 \frac{1}{A_{s}(\ct (\phi ))}=\frac{1}{A_{s}(\phi )}
 +\frac{e^{2\phi _{t}}-1}{A_{t}(\phi )},\quad 0<s\le t.
\end{align*}
Noting relation~\eqref{;deria}, we differentiate both sides with respect to $s$ to get 
\begin{align*}
 \bigl\{ 
 Z_{s}(\ct (\phi ))
 \bigr\} ^{-2}=\{ Z_{s}(\phi )\} ^{-2},\quad 0<s\le t,
\end{align*}
which proves the claim by the positivity of $Z$.

\thetag{ii} We argue similarly to the reasoning in \rref{;rinvol}, 
appealing to relations~\eqref{;iipd} and \eqref{;impd} with $c=1$.
Note that $\ct =\ct ^{(1)}$ by definition. For every $\phi \in C([0,t];\R )$ 
and $0<s\le t$, we have, thanks to property~\thetag{i}, 
\begin{align*}
 (\ct \circ \ct )(\phi )(s)=-\log\left\{ 
 Z_{s}(\phi )\int _{s}^{t}\frac{du}{\left( 
 Z_{u}(\phi )
 \right) ^{2}}+\frac{Z_{s}(\phi )}{Z_{t}(\phi )}e^{\ct (\phi )(t)}
 \right\} ,
\end{align*}
which is seen to agree with $\phi _{s}$ by noting that $\ct (\phi )(t)=-\phi _{t}$.

\thetag{iii} For every $s\in [0,t]$, by the definition~\eqref{;trev} of $R$, 
\begin{align*}
 (R\circ \ct )(\phi )(s)=\phi _{t-s}-\log \left\{ 
 1+\frac{A_{t-s}(\phi )}{A_{t}(\phi )}\left( 
 e^{2\phi _{t}}-1
 \right) 
 \right\} +\phi _{t},
\end{align*}
whereas 
\begin{align*}
 (\ct \circ R)(\phi )(s)=\phi _{t-s}-\phi _{t}-\log \left\{ 
 1+\frac{A_{t}(\phi )-A_{t-s}(\phi )}{A_{t}(\phi )}\left( 
 e^{-2\phi _{t}}-1
 \right) 
 \right\} ,
\end{align*}
in which we have the equality 
\begin{align*}
 1+\frac{A_{t}(\phi )-A_{t-s}(\phi )}{A_{t}(\phi )}\left( 
 e^{-2\phi _{t}}-1
 \right) 
 =e^{-2\phi _{t}}\left\{ 
 1+\frac{A_{t-s}(\phi )}{A_{t}(\phi )}\left( 
 e^{2\phi _{t}}-1
 \right) 
 \right\} .
\end{align*}
Combining these leads to the conclusion.
\end{proof}

\subsection{Proof of \eqref{;invw}}

With slight abuse of notation, we simply 
write $Z_{s}=Z_{s}(B),\,s\ge 0$, in the sequel. We begin with the following 
observation asserting that the conditional law of $B_{t}$ 
given $\sigma (Z_{s},s\le t)$ is symmetric.

\begin{lem}\label{;lcsym}
 It holds that 
 \begin{align*}
 \left( 
 \{ Z_{s}\} _{0\le s\le t},\,B_{t}
 \right) \eqd 
 \left( 
 \{ Z_{s}\} _{0\le s\le t},\,-B_{t}
 \right) .
 \end{align*}
\end{lem}

\begin{proof}
The assertion is an immediate consequence of the fact 
\cite[Proposition~1.7]{myPI} that, given $\sigma (Z_{s},s\le t)$ 
with $Z_{t}=z>0$, the conditional law of $B_{t}$ admits the 
symmetric density function  
\begin{align*}
 \frac{1}{2K_{0}(1/z)}\exp \left( 
 -\frac{\cosh x}{z}
 \right) ,\quad x\in \R ,
\end{align*}
with respect to the Lebesgue measure. 
Here $K_{0}$ is the modified Bessel function of the third kind 
(or the Macdonald function) of order $0$; see, e.g., \cite[Section~5.7]{leb} 
for its definition.
\end{proof}

Thanks to the Cameron--Martin formula and the injectivity of the Mellin 
transform, the above conditional symmetry of $B_{t}$ may also be deduced from 
the fact \cite[Theorem~1.6\thetag{ii}]{myPI} that, for any $\mu \in \R $, 
\begin{align*}
 \bigl\{ 
 Z_{s}(\db{\mu })
 \bigr\} _{s\ge 0}
 \eqd 
 \bigl\{ 
 Z_{s}(\db{-\mu })
 \bigr\} _{s\ge 0},
\end{align*}
which is an exponential analogue to identity~\eqref{;pinv} as discussed in 
\cite{myPI}.

We may compare \lref{;lcsym} with identity~\eqref{;csym}. The proof of \eqref{;invw} 
proceeds similarly to that of \eqref{;invm} exhibited in \sref{;sppexpress}.

\begin{proof}[Proof of \eqref{;invw}]
By relation~\eqref{;iipd} with $c=1$, 
\begin{align*}
 B_{s}=-\log \left( 
 Z_{s}\int _{s}^{t}\frac{du}{Z_{u}^{2}}+\frac{Z_{s}}{Z_{t}}e^{-B_{t}}
 \right) 
\end{align*}
for $0<s\le t$. Therefore, by \lref{;lcsym}, 
\begin{align*}
 \{ B_{s}\} _{0<s\le t}\eqd 
 \left\{ 
 -\log \left( 
 Z_{s}\int _{s}^{t}\frac{du}{Z_{u}^{2}}+\frac{Z_{s}}{Z_{t}}e^{B_{t}}
 \right) 
 \right\} _{0<s\le t},
\end{align*}
the right-hand side of which is identical with $\{ \ct (B)(s)\} _{0<s\le t}$ 
in view of \eqref{;impd}.
\end{proof}

\subsection{Two fundamental lemmas}\label{;sstfl}

The following lemma has been used in the proofs of \lref{;lconvs} and 
\pref{;pexpress}.

\begin{lem}\label{;lfund}
Let $\{ a_{n}\} _{n=1}^{\infty }$, $\{ b_{n}\} _{n=1}^{\infty }$ and 
$\{ c_{n}\} _{n=1}^{\infty }$ be three sequences of positive real numbers 
such that, as $n\to \infty $, 
\begin{align*}
 \frac{1}{c_{n}}\log a_{n}\to \al , &&  
 \frac{1}{c_{n}}\log b_{n}\to \beta , && 
 c_{n}\to \infty ,
\end{align*}
for some $\al ,\beta \in \R $. Then it holds that 
\begin{align*}
 \lim _{n\to \infty }\frac{1}{c_{n}}\log (a_{n}+b_{n})=\max\{ \al ,\beta \}. 
\end{align*}
\end{lem}

\begin{proof}
Without loss of generality, we may assume $\al \ge \beta $. Fix $\ve >0$ arbitrarily. 
By assumption, there exists a positive integer $N$ such that, for all $n\ge N$, 
\begin{align*}
 e^{(\al -\ve )c_{n}}<a_{n}<e^{(\al +\ve )c_{n}}, 
 && e^{(\beta -\ve )c_{n}}<b_{n}<e^{(\beta +\ve )c_{n}},
\end{align*}
from which we have 
\begin{align*}
 \al -\ve +\frac{1}{c_{n}}\log \left\{ 
 1+e^{-(\al -\beta )c_{n}}
 \right\} &<\frac{1}{c_{n}}\log (a_{n}+b_{n})\\
 &<\al +\ve +\frac{1}{c_{n}}\log \left\{ 
 1+e^{-(\al -\beta )c_{n}}
 \right\} .
\end{align*}
Therefore it follows that, by the assumption that $c_{n}\to \infty $ as $n\to \infty $, 
\begin{align*}
 \al -\ve &\le \liminf _{n\to \infty }\frac{1}{c_{n}}\log (a_{n}+b_{n})\\
 &\le \limsup _{n\to \infty }\frac{1}{c_{n}}\log (a_{n}+b_{n})\le \al +\ve ,
\end{align*}
which proves the claim as $\ve >0$ is arbitrary. 
\end{proof}

The following lemma is referred to in \rref{;rlconvs}, as well as in the 
proofs of \lref{;lconvs} and \pref{;pexpress}.

\begin{lem}\label{;tdini}
Given $a,b\in \R $ with $a<b$, let $\{ f_{n}\} _{n=1}^{\infty }$ be a 
sequence of real-valued nondecreasing functions on $[a,b]$ such that 
\begin{align*}
 \lim _{n\to \infty }f_{n}(s)=f(s)\quad \text{for all $s\in [a,b]$},
\end{align*}
with some continuous function $f$ on $[a,b]$. Then it holds that 
\begin{align*}
 \lim _{n\to \infty }\sup _{s\in [a,b]}\left| 
 f_{n}(s)-f(s)
 \right| =0.
\end{align*}
The same conclusion holds true if each $f_{n}$ is nonincreasing on $[a,b]$.
\end{lem}

\begin{proof}
We deal with the case that each $f_{n}$ is nondecreasing so that the 
limit function $f$ is also nondecreasing on $[a,b]$. Fix $\ve >0$ arbitrarily. 
By the uniform continuity of $f$, we may divide $[a,b]$ into 
subintervals $[x_{k-1},x_{k}],\,k=1,\ldots ,N$, so that 
\begin{align}\label{;qdini1}
 f(x_{k})-f(x_{k-1})<\ve ,\quad k=1,\ldots ,N.
\end{align}
Moreover, we may take $n$ so large that 
\begin{align*}
 \max _{0\le k\le N}\left| 
 f_{n}(x_{k})-f(x_{k})
 \right| <\ve .
\end{align*}
For every $x\in [a,b]$, pick a subinterval $[x_{k-1},x_{k}]$ containing $x$. 
Then we have, by the assumption that $f_{n}$ is nondecreasing, 
\begin{align*}
 f(x_{k-1})-\ve <f_{n}(x_{k-1})\le f_{n}(x)\le f_{n}(x_{k})<f(x_{k})+\ve .
\end{align*}
Because 
\begin{align*}
 f(x_{k})<f(x)+\ve && \text{and} && f(x)-\ve <f(x_{k-1})
\end{align*}
by \eqref{;qdini1}, we obtain $\left| f_{n}(x)-f(x)\right| <2\ve $. 
\end{proof}


\end{document}